\documentclass[12pt]{amsart}
\usepackage{a4wide}
\usepackage[utf8]{inputenc}
\usepackage{amsmath}
\usepackage{amssymb}
\usepackage{amsfonts}
\usepackage{amsopn}
\usepackage{graphicx}
\usepackage{enumerate}
\usepackage{color}
\usepackage{mathtools}
\usepackage{mathrsfs}
\usepackage{bbm}
\usepackage{yhmath}
\usepackage{tikz}
\usepackage[colorlinks,linkcolor=blue,anchorcolor=blue,citecolor=blue]{hyperref}

\newtheorem{theorem}{Theorem}[section]
\newtheorem{proposition}[theorem]{Proposition}
\newtheorem{lemma}[theorem]{Lemma}

\newtheorem{corollary}[theorem]{Corollary}
\newtheorem{question}{Question}
\newtheorem{conjecture}[question]{Conjecture}
\newtheorem*{theoremA}{Theorem A}
\theoremstyle{definition}

\theoremstyle{remark}
\newtheorem{remark}[theorem]{Remark}

\numberwithin{equation}{section}

\newcommand{\R}{\mathbb{R}}
\newcommand{\Q}{\mathbb{Q}}
\newcommand{\Z}{\mathbb{Z}}
\newcommand{\N}{\mathbb{N}}
\newcommand{\PP}{\mathbb{P}}

\newcommand{\f}{\infty}

\begin{document}

\title[Infinite scalings of the canonical spectrum]{On infinite scalings of the canonical spectrum for self-similar spectral measures}

\author{Zhiqiang Wang}
\address{College of Mathematics and Statistics, Center of Mathematics, Chongqing University, Chongqing 401331, People's Republic of China \& Department of Mathematics, University of British Columbia, Vancouver, British Columbia, V6T 1Z2, Canada}
\email{zhiqiangwzy@163.com,~zqwangmath@cqu.edu.cn}


\begin{abstract}
Let $(\mu, \Lambda)$ be the canonical spectral pair generated by a Hadamard triple $(N,B,L)$ in $\R$ with $0\in B \cap L$, which means that the family $\big\{ e_\lambda(x)=e^{2\pi \mathrm{i} \lambda x}: \lambda \in \Lambda \big\}$ forms an orthonormal basis in $L^2(\mu)$.
We prove that if $\#B < N^{0.677}$, then there are infinitely many primes $p$ such that $(\mu, p\Lambda)$ is also a spectral pair.
Under Artin's primitive root conjecture or the Elliott–Halberstam conjecture, the same conclusion holds for $\# B < N$.
\end{abstract}

\maketitle

\section{Introduction}

\subsection{Spectral eigenvalue problems}
A Borel probability measure $\mu$ on $\R^d$ is called a \emph{spectral measure} if there exists a countable set $\Lambda \subset \R^d$ such that $$ \big\{ e_\lambda(x)=e^{2\pi \mathrm{i} \lambda \cdot x}: \lambda \in \Lambda \big\} $$forms an orthonormal basis in $L^2(\mu)$. The set $\Lambda$ is called a \emph{spectrum} of $\mu$, and the pair $(\mu,\Lambda)$ is called a \emph{spectral pair}.
It is well-known that the Lebesgue measure on the unit hypercube $[0,1]^d$ and the integer lattice $\Z^d$ constitute a spectral pair.
The study of spectral measures dates back to B. Fuglede, who raised the famous spectral set conjecture \cite{Fuglede-1974}.
In 1998, Jorgensen and Pedersen \cite{Jorgensen-Pedersen-1998} found a class of singular continuous self-similar spectral measures.
For example, the self-similar measure $\mu_{4,\{0,2\}}$, which satisfies the equality
\[ \mu(\;\cdot\;) = \frac{1}{2} \mu(4\;\cdot\;) + \frac{1}{2}\mu(4\;\cdot\; -2),\]
has a spectrum \[ \Lambda_{4,\{0,1\}} = \bigg\{ \sum_{k=0}^{n} \ell_k 4^k:\; n \in \N,\; \ell_0, \ell_1, \ldots,\ell_n \in \{0,1\}  \bigg\}. \]
Since then, a wide variety of singular continuous spectral measures have been constructed, see \cite{Strichartz-2000,Laba-Wang-2002,An-He-2014,An-Fu-Lai-2019,Dutkay-Haussermann-Lai-2019,
Li-Miao-Wang-2022,Lu-Dong-Zhang-2022,An-Lai-2023a,WuH-2024} and the references therein.

An interesting phenomenon occurring in singular continuous spectral measures is the scaling spectrum, which was first discovered by Strichartz \cite[Example 2.9]{Strichartz-2000} and independently by {\L}aba and Wang \cite[Example 3.2]{Laba-Wang-2002}.
They found that some singular continuous spectral measures have both $\Lambda$ and its scaling $t\Lambda$ for some $t>1$ as its spectrum.
A more compelling aspect is that distinct spectra may correspond to varying convergence and divergence properties in mock Fourier series \cite{Strichartz-2006,Dutkay-Han-Sun-2014,Fu-Tang-Wen-2022,Pan-Ai-2023}.
These surprising phenomenons motivate the following two questions.

\begin{question}\label{Q-1}
  For a singular continuous spectral pair $(\mu,\Lambda)$ in $\R$, to find all real number $t \in \R$ such that $(\mu,t\Lambda)$ is also a spectral pair.
\end{question}

\begin{question}\label{Q-2}
  For a singular continuous spectral measure $\mu$ in $\R$, to find all real number $t \in \R$ such that both $(\mu,\Lambda)$ and $(\mu,t\Lambda)$ are spectral pairs for some $\Lambda \subset \R$.
\end{question}

The real number $t \in \R$ in \textbf{Question \ref{Q-1}} and \textbf{Question \ref{Q-2}} is called a \emph{spectral eigenvalue} of the spectral pair $(\mu,\Lambda)$ and the spectral measure $\mu$, respectively.
Sometimes, the real number $t \in \R$ in \textbf{Question \ref{Q-1}} is also called a \emph{complete number} of the spectral pair $(\mu,\Lambda)$.

On \textbf{Question \ref{Q-1}}, significant attention has been directed toward the self-similar spectral measures and its canonical spectrum \cite{Jorgensen-Kornelson-Shuman-2011,Li-2011,Li-Xing-2017,Dutkay-Kraus-2018,He-Tang-Wu-2019,Li-Wu-2022,
Jiang-Lu-Wei-2024,Yi-Zhang-2025a,Chi-He-Wu-2025}, especially the spectral pair $(\mu_{4,\{0,2\}},\Lambda_{4,\{0,1\}})$ \cite{Dutkay-Jorgensen-2012,Dai-2016,
Dutkay-Haussermann-2016}. Dutkay and Haussermann \cite{Dutkay-Haussermann-2016} showed that for any prime $p > 3$, the integer $p^n$ is a spectral eigenvalue of $(\mu_{4,\{0,2\}}, \Lambda_{4,\{0,1\}})$ for all $n \in \N$, and Dai \cite{Dai-2016} characterized all integer spectral eigenvalues.
For the self-similar spectral measure with three digits and its canonical spectrum, He, Tang and Wu \cite{He-Tang-Wu-2019} established an equivalent condition for all spectral eigenvalues, which implies that any given real number can be determined in a finite number of steps whether it is a spectral eigenvalue.
They remarked that the analogous result also holds for $(\mu_{4,\{0,2\}},\Lambda_{4,\{0,1\}})$ by appropriate adjustments.
Recently, Chi, He and Wu \cite{Chi-He-Wu-2025} investigated \textbf{Question \ref{Q-1}} in the setting of $N$-Bernoulli convolutions.
In general, there is a folklore conjecture, see also \cite[Conjecture 6.2]{Dai-Fu-He-2026} and \cite[Conjecture 1.2]{Chi-He-Wu-2025}.

\begin{conjecture}\label{C-3}
  There are infinitely many spectral eigenvalues of a singular continuous spectral pair in $\R$.
\end{conjecture}

Regarding \textbf{Question} \ref{Q-2}, Fu, He, and Wen \cite{Fu-He-Wen-2018} first characterized all spectral eigenvalues of spectral Bernoulli convolutions (including the particular case $\mu_{4,\{0,2\}}$).
This result was extended by Fu and He \cite{Fu-He-2017} to self-similar spectral measures with consecutive digits and some Moran spectral measures.
For self-similar spectral measures with product-form digit set, spectral eigenvalues  were also characterized in \cite{Li-Wu-2022,Jiang-Lu-Wei-2024}.
In higher dimensions, spectral eigenmatrices  for some self-affine spectral measures have been investigated as well \cite{An-Dong-He-2022,Chen-Liu-2023,Liu-Tang-Wu-2023,Liu-Liu-Tang-Wu-2024,Chen-Cao-2025}.

In \cite{Kong-Li-Wang-2025}, Kong, Li and the author explored \textbf{Question \ref{Q-2}} and provided a class of spectral eigenvalues, which is dense in $[0,+\f)$, for any singular continuous self-similar spectral measure generated by a Hadamard triple.
In this paper, the author will deal with \textbf{Conjecture \ref{C-3}} for singular continuous spectral pairs generated by a Hadamard triple, and will establish connections with number theory.

\subsection{Main results}
Let $\N =\{1,2,3,\ldots\}$, and for $k \in \N$, let $\N_{\ge k} = \N \cap [k,+\f)$.
Let $\#A$ denote the cardinality of a set $A$.

Let $N \in \N_{\ge 2}$ and let $B,L \subset \Z$ be finite sets with $\# B = \#L \ge 2$.
We say $(N,B,L)$ is a \emph{Hadamard triple} in $\R$ if the matrix
\begin{equation}\label{eq:matrix}
  \bigg( \frac{1}{\sqrt{\# B}} e^{2\pi \mathrm{i} \frac{b\ell}{N}} \bigg)_{b \in B, \ell \in L}
\end{equation}
is unitary.
Associated with a Hadamard triple $(N,B,L)$, there are two \emph{iterated function systems} (IFS):
$$\mathcal{F}_{N,B}=\bigg\{ \tau_b(x) = \frac{x+b}{N}: b \in B \bigg\}\;\;\text{and}\;\; \mathcal{F}_{N,L}= \bigg\{ \tau_\ell(x) = \frac{x+\ell}{N}: \ell \in L \bigg\}.$$
According to Hutchinson \cite{Hutchinson-1981}, there exists a unique Borel probability measure $\mu_{N,B}$, which is called a \emph{self-similar measure}, such that $$\mu = \frac{1}{\#B} \sum_{b \in B} \mu\circ\tau_b^{-1}.$$
{\L}aba and Wang \cite{Laba-Wang-2002} proved that the self-similar measure $\mu_{N,B}$ associated with a Hadamard triple $(N,B,L)$ in $\R$ is a spectral measure.
This result was generalized to the self-affine measure in higher dimensions by Dutkay, Haussermann, and Lai \cite{Dutkay-Haussermann-Lai-2019}.

Let $(N,B,L)$ be a Hadamard triple in $\R$ with $0 \in B \cap L$.
Define $$m_B(x) = \frac{1}{\# B} \sum_{b \in B} e^{2\pi \mathrm{i} b x}.$$
A finite set $C=\{x_1,x_2,\ldots, x_k\}$ is called a \emph{cycle} if there exist $\ell_1,\ell_2,\ldots \ell_k \in L$ such that \[ \tau_{\ell_j}(x_j) = x_{j+1} \;\text{for}\; j=1,2,\ldots, k-1, \;\text{and}\; \tau_{\ell_k}(x_k) = x_1.  \]
The cycle $C$ is called an \emph{$m_B$-cycle} (or \emph{extreme cycle}) if $|m_B(x_j)|=1$ for all $1\le j \le k$.
Note that $\{0\}$ is a trivial $m_B$-cycle.
Dutkay and Jorgensen \cite{Dutkay-Jorgensen-2008} gave the following construction of spectrum for self-similar spectral measures, see \cite[Theorem 3.9]{Dutkay-Lai-2017} for an independent proof.

\begin{theoremA}{\rm{\cite[Theorem 2.38]{Dutkay-Jorgensen-2008}}}
  \label{thm:spectrum}
  Let $(N,B,L)$ be a Hadamard triple in $\R$ with $0\in B \cap L$.
  Let $\Lambda$ be the smallest set that contains $-C$ for all $m_B$-cycles $C$, and that satisfies $N \Lambda + L \subset \Lambda$.
  Then the set $\Lambda$ is a spectrum of the self-similar measure $\mu_{N,B}$.
\end{theoremA}

The spectral pair $(\mu_{N,B},\Lambda)$ in \textbf{Theorem A} is called a \emph{canonical spectral pair} generated by a Hadamard triple $(N,B,L)$ in $\R$.
The unitarity of matrix in (\ref{eq:matrix}) implies that all elements in $B$ are distinct modula $N$. Thus, we have $\# B \le N$.
If $\#B =N$, then the self-similar spectral measure $\mu_{N,B}$ is absolutely continuous.
We focus on the case $\# B < N$.

Let $\pi(x)$ be the prime-counting function, i.e., the number of primes not exceeding $x$.
By the well-known prime number theory, we have \[ \pi(x) \sim \frac{x}{\log x} \;\;\text{as}\; x \to +\f. \]
Our main result is to show that if $\# B$ is small relative to $N$, then there are infinitely many prime spectral eigenvalues of canonical spectral pairs.

\begin{theorem}\label{thm-main-1}
  Let $(\mu_{N,B}, \Lambda)$ be the canonical spectral pair generated by a Hadamard triple $(N,B,L)$ in $\R$ with $0\in B \cap L$.

  {\rm(i)} If $\# B < N^{0.677}$, then there are infinitely many prime spectral eigenvalues of $(\mu_{N,B},\Lambda)$. Furthermore,
  $$\liminf_{x \to +\f} \frac{\# \big\{ p \le x: \; p \;\text{is a prime and a spectral eigenvalue of}\; (\mu_{N,B},\Lambda) \big\}}{\pi(x)} >0.$$

  {\rm(ii)} If $\# B < N^{1/2}$, then almost all primes are spectral eigenvalues of $(\mu_{N,B},\Lambda)$ in the sense that $$\lim_{x \to +\f} \frac{\# \big\{ p \le x: \; p \;\text{is a prime and a spectral eigenvalue of}\; (\mu_{N,B},\Lambda) \big\}}{\pi(x)} = 1.$$
\end{theorem}
\begin{remark}
  It is worth mentioning that Theorem \ref{thm-main-1} (ii) for $N$-Bernoulli convolutions has been obtained by Chi, He and Wu \cite{Chi-He-Wu-2025}.
\end{remark}

When $\# B$ is close to $N$, we present a conditional proof of \textbf{Conjecture \ref{C-3}}.
The details of Artin's primitive root conjecture and the Elliott–Halberstam conjecture are given in Section \ref{sec:connection}.

\begin{theorem}\label{thm:conditional}
  Under Artin's primitive root conjecture or the Elliott–Halberstam conjecture, \textbf{Conjecture \ref{C-3}} holds for singular continuous canonical spectral pairs generated by a Hadamard triple in $\R$.
\end{theorem}

For composite numbers to be spectral eigenvalues, we give the following sufficient condition.
Analogous results for some special spectral pairs also appear in \cite{Dutkay-Haussermann-2016,Dutkay-Kraus-2018,Chi-He-Wu-2025}.

\begin{theorem}\label{thm-main-2}
  Let $(\mu_{N,B}, \Lambda)$ be the canonical spectral pair generated by a Hadamard triple $(N,B,L)$ in $\R$ with $0\in B \cap L$ and $\# B < N$.
  Let $p_1, p_2, \ldots, p_k$ be distinct primes such that $\gcd(p_1 p_2 \cdots p_k, N) =1$.
  Then there exists a large enough integer $n_0$, depending on $(N,B,L)$ and $p_1, p_2, \ldots, p_k$, such that if $(p_1 p_2 \cdots p_k)^{n_0}$ is a spectral eigenvalue of $(\mu_{N,B}, \Lambda)$ then $p_1^{n_1} p_2^{n_2} \cdots p_k^{n_k}$ is a spectral eigenvalue of $(\mu_{N,B}, \Lambda)$ for all non-negative integers $n_1, n_2, \ldots, n_k$.
\end{theorem}

Let $(N,B,L)$ be a Hadamard triple in $\R$ with $0\in B \cap L$ and $N>\#B \ge 2$.
If $\{0\}$ is the only $m_B$-cycle, then the canonical spectrum in \textbf{Theorem A} is given by
\begin{equation}\label{eq:Lambda-N-L}
  \Lambda_{N,L} = \bigg\{ \sum_{k=0}^{n} \ell_k N^k:\; n \in \N,\; \ell_0, \ell_1, \ldots,\ell_n \in L  \bigg\},
\end{equation}
which happens if $\gcd(B)=1$ and $L \subset [2-N,N-2]$ \cite[Theorem 1.2]{Laba-Wang-2002}.
Applying Theorem \ref{thm-main-1} (i) to the case that $\#B =2,3,4$, we obtain the following corollary.
\begin{corollary}\label{cor:2-3-4}
  There are infinitely many prime spectral eigenvalues of $(\mu_{N,B},\Lambda_{N,L})$, where $\mu_{N,B}$ is the self-similar measure and $\Lambda_{N,L}$ is defined in (\ref{eq:Lambda-N-L}), for all following cases.

  {\rm(i)} $N=2q$ with $q \in \N_{\ge 2}$, $B=\{0,1\}$, and $L=\{ 0,q \} \;\text{or}\; \{-q,0\}$.

  {\rm(ii)} $N=3q$ with $q \in \N_{\ge 2}$, $B=\{0,a,b\} \subset \Z$ with $\gcd(a,b)=1$ and $\{a,b\} \equiv \{1,2\} \pmod{3}$, and $L=\{0,q,2q\},\;\{-q,0,q\},\;\text{or}\;\{-2q,-q,0\}$.

  {\rm(iii)} $B=\big\{ 0, a, 2^m b, a+2^m b' \big\}$ where $m \in \N$, $a,b,b' \in \N$ are odd integers and $\gcd(a,b,b')=1$, $N \in \N_{\ge 8}$ with $2^{m+1} \mid N$, and $L=\frac{N}{2^{m+1}}\big\{ 0, 1, 2^m, 1+2^m \big\}$.
\end{corollary}
\begin{remark}
  We remark that Corollary \ref{cor:2-3-4} (ii) offers an affirmative answer to Question 6.3 in \cite{Dai-Fu-He-2026}.
\end{remark}

The rest of paper is organized as follows.
In Section \ref{preliminary}, we give some needed results for subsequent proofs.
We prove Theorems \ref{thm-main-1} and \ref{thm-main-2} in Section \ref{sec:proof}.
In the last section, we discuss Artin's primitive root conjecture and the Elliott–Halberstam conjecture, and prove Theorem \ref{thm:conditional}.

\section{Preliminaries}\label{preliminary}

In this section, we collect some known results for subsequent proofs.

\subsection{Rational points in Cantor sets}
Let $q \in \N_{\ge 3}$, and let $A \subset \Q$ be a finite set.
Consider the IFS \[{\mathcal F}_{q,A}=\bigg\{ f_a(x)=\frac{x+a}{q}: a\in A \bigg\},\]
and by Hutchinson \cite{Hutchinson-1981}, there exists a unique non-empty compact set $K(q,A)$ in $\R$, which is called a \emph{self-similar set}, such that \[ K(q,A) = \bigcup_{a \in A} f_a\big( K(q,A) \big). \]
The set $K(q, A)$ can be written algebraically as
\[ K(q,A)=\bigg\{\sum_{k=1}^{\infty} \frac{a_k}{q^{k}}: \;a_k \in A ~~\forall k\in \N\bigg\}.\]
Every point $x \in K(q,A)$ has at least one expression of the form $x = \sum_{k=1}^{\infty} a_k q^{-k}$ with each $a_k \in A$. The infinite sequence $(a_k)_{k=1}^\f \in A^\N$ is called a \emph{coding} of $x$.
Let $\dim_{\mathrm{H}} E$ denote the Hausdorff dimension of a set $E$.

\begin{lemma}\cite[Lemma 2.3]{Kong-Li-Wang-2025}\label{lemma:period}
Let $q\in\N_{\ge 3}$ and let $A  \subset \mathbb Q$ be a finite set. If $s=\dim_{\mathrm{H}}K(q, A) <1 $, then there exists a constant $c>0$ such that any rational number $\xi = v/u \in K(q, A)$ with $u \in \N$, $v \in \Z$ and $\gcd(v,u)=1$ admits an eventually periodic coding with the periodic length $\le c u^{s}.$
\end{lemma}

For $p\in \N_{\ge 2}$, let $D_p$ be the set of all rational numbers in $\mathbb {R}$ having a finite $p$-ary expansion. That is,
\[ D_p=\bigcup_{k=0}^\f \frac{\mathbb {Z}}{p^k}, \]
which is a proper subset of $\mathbb Q$ and is dense in $\mathbb R$.

\begin{theorem}\cite[Theorem 1.1]{Kong-Li-Wang-2025}\label{thm:finiteness}
  Let $q\in\N_{\ge 3}$ and let $A  \subset \mathbb Q$ be a finite set.
  If $\dim_{\mathrm{H}} K(q,A) < 1$, then for any $p \in \N_{\ge 2}$ with $\gcd(p,q) = 1$ we have \[ \# \big( D_p \cap K(q,A) \big) < +\f.\]
\end{theorem}

\subsection{Canonical spectral pair construction}
The canonical spectrum $\Lambda$ in \textbf{Theorem A} can be constructed explicitly.
As before, let $K(N,L)$ be the self-similar set generated by the IFS $\mathcal{F}_{N,L}=\big\{ \tau_\ell(x) = (x+\ell)/N: \ell \in L \big\}$.
\begin{proposition}\cite[Proposition 3.6]{Kong-Li-Wang-2025}
\label{prop:construct-spectrum}
  Let $(N,B,L)$ be a Hadamard triple in $\R$ with $0\in B \cap L$ and write $d = \gcd(B)$.
  Let $\Lambda_0= -\big( K(N,L) \cap (\Z/d) \big)$ and $\Lambda_n = N \Lambda_{n-1} + L$ for $n \in \N$.
  Set $$\Lambda = \bigcup_{n=0}^\f \Lambda_n.$$
  Then $(\mu_{N,B},\Lambda)$ is the canonical spectral pair generated by the Hadamard triple $(N,B,L)$.
\end{proposition}

\subsection{The order of $a$ module $p$}
Let $\PP$ be the set of primes and let $a \in \N_{\ge 2}$.
For $p \in \PP$ with $p \nmid a$, the \emph{order of $a$ modula $p$} is defined to be the smallest positive integer satisfying $a^n \equiv 1 \pmod{p}$ and is denoted by $\mathrm{Ord}_a(p)$.
For $0< \delta < 1$, define
\begin{equation}\label{eq:A-a-delta}
  A_a(\delta) = \big\{ p \in \PP: p \nmid a \;\text{and}\; \mathrm{Ord}_a(p) > p^{\delta} \big\}.
\end{equation}
It is clear that if $0< \delta_1 < \delta_2 < 1$ then we have $A_a(\delta_2) \subset A_a(\delta_1)$.
In 1976, Erd\H{o}s \cite{Erdos-1976} proved that \[ \lim_{x \to +\f} \frac{\#\big( A_2(1/2) \cap [0,x] \big)}{\pi(x)} = 1. \]
Erd\H{o}s's techniques can be used to prove similar results for any $a \in \N_{\ge 2}$.
\begin{theorem}\cite{Erdos-1976}\label{thm:Erdos-1976}
  For any $a \in \N_{\ge 2}$, we have \[ \lim_{x \to +\f} \frac{\#\big( A_a(1/2) \cap [0,x] \big)}{\pi(x)} = 1. \]
\end{theorem}


\subsection{Shifted primes with large prime factors}
Let $P^+(n)$ denote the largest prime factor of a positive integer $n$ with the convention that $P^+(1) = 1$.
In 1969, Goldfeld \cite{Goldfeld-1969} showed that
\begin{equation}\label{eq:Goldfeld-1969}
  \liminf_{x \to +\f} \frac{\#\big\{ p \le x: \; p \in \PP,\; P^+(p-1) > x^{1/2} \big\}}{\pi(x)} \ge \frac{1}{2}.
\end{equation}
by using the Bombieri-Vinogradov theorem and the Brun-Titchmarsh theorem.

For $a \in \N_{\ge 2}$ and $p \in \PP$ with $p \nmid a$, Fermat's little theorem states that \[ a^{p-1} \equiv 1 \pmod{p}. \]
It follows that \[ \mathrm{Ord}_a(p)  \mid  p-1. \]
Goldfeld \cite{Goldfeld-1969} also proved that for almost all primes $p$ for which $p-1$ has a large prime factor $q$, then $q$ also divides $\mathrm{Ord}_a(p)$, which relates (\ref{eq:Goldfeld-1969}) to the set $A_a(\delta)$.
The following lemma is incorporated within the proof of Theorem 2 in \cite{Goldfeld-1969}.

\begin{lemma}\label{lemma:Goldfeld-1969}
  For any $a \in \N_{\ge 2}$, we have
  \[ \lim_{x \to +\f} \frac{\#\big\{ p\le x: p \in \PP, \;\exists\; q\in \PP \;\text{such that}\; q>x^{1/2},\; q \mid (p-1) \;\text{but}\; q \nmid \mathrm{Ord}_a(p) \}}{\pi(x)} = 0. \]
\end{lemma}

The exponent $1/2$ in (\ref{eq:Goldfeld-1969}) has been improved to $0.677$ by Baker and Harman \cite{Baker-Harman-1995,Baker-Harman-1998}, which is the best record up to now.

\begin{theorem}\cite{Baker-Harman-1995,Baker-Harman-1998}
  We have
  \[ \liminf_{x \to +\f} \frac{\#\big\{ p \le x: \; p \in \PP,\; P^+(p-1) > x^{0.677} \big\}}{\pi(x)} >0. \]
\end{theorem}

Together with Lemma \ref{lemma:Goldfeld-1969}, we obtain the following corollary.

\begin{corollary}\label{cor:Baker-Harman}
  For any $a \in \N_{\ge 2}$, we have \[ \lim_{x \to +\f} \frac{\#\big( A_a(0.677) \cap [0,x] \big)}{\pi(x)} >0. \]
  In particular, the set $A_a(0.677)$ is infinite.
\end{corollary}

\section{Proofs of Theorems \ref{thm-main-1} and \ref{thm-main-2}}\label{sec:proof}
In this section, we always assume that $(\mu_{N,B}, \Lambda)$ is the canonical spectral pair generated by a Hadamard triple $(N,B,L)$ in $\R$ with $0\in B \cap L$ and $\# B < N$, and write $d = \gcd(B)$.
We first give a characterization of spectral eigenvalues.

\begin{proposition}\label{prop:eigenvalue-equivalent}
  For $q \in \N_{\ge 2}$, if $\gcd(q,N) = 1$, then $q$ is a spectral eigenvalue of $(\mu_{N,B}, \Lambda)$ if and only if
  \begin{equation}\label{eq:eigenvalue-condition}
    K(N, L) \cap \frac{\Z}{qd} = K(N,L) \cap \frac{\Z}{d}.
  \end{equation}
\end{proposition}

We need the following lemma about Hadamard triples to prove the proposition.

\begin{lemma}\cite[Lemma 3.4]{Kong-Li-Wang-2025}\label{lemma:multiply-q}
  Let $(N,B,L)$ be a Hadamard triple in $\R$.
  If $q \in \N_{\ge 2}$ with $\gcd(q,N)=1$, then $(N,B,qL)$ is also a Hadamard triple.
\end{lemma}

\begin{proof}[Proof of Proposition \ref{prop:eigenvalue-equivalent}]
By Proposition \ref{prop:construct-spectrum}, the set $\Lambda$ is constructed as follows: let
  \begin{equation}\label{eq:Lambda-0}
    \Lambda_0 = - \Big( K(N,L) \cap \frac{\Z}{d}  \Big),
  \end{equation}
  and $\Lambda_n = N \Lambda_{n-1} + L$ for $n \in \N$,
  then $$\Lambda = \bigcup_{n=0}^\f \Lambda_n.$$

Since $\gcd(q,N) = 1$, by Lemma \ref{lemma:multiply-q}, we have $(N,B,qL)$ is also a Hadamard triple.
Applying Proposition \ref{prop:construct-spectrum} for the Hadamard triple $(N,B,qL)$, we obtain a spectrum $\Lambda'$ of $\mu_{N,B}$ as follows: let
  \begin{equation}\label{eq:Lambda-0-prime}
    \Lambda'_0 = - \Big( K(N,qL) \cap \frac{\Z}{d} \Big),
  \end{equation}
  and $\Lambda'_n = N \Lambda'_{n-1} + q L$ for $n \in \N$,
  then the set $$\Lambda' = \bigcup_{n=0}^\f \Lambda'_n$$
  is a spectrum of $\mu_{N,B}$.
By (\ref{eq:Lambda-0}) and (\ref{eq:Lambda-0-prime}), we have
\[ \Lambda'_0 = - q\Big( K(N,L) \cap \frac{\Z}{qd} \Big) \supset - q\Big( K(N,L) \cap \frac{\Z}{d} \Big) = q \Lambda_0. \]
By recursion, we obtain $\Lambda'_n = N \Lambda'_{n-1} + q L \supset q\big( N \Lambda_{n-1} + L \big) = q \Lambda_n$ for any $n \in \N$.
It follows that $$\Lambda' = \bigcup_{n=0}^\f \Lambda'_n \supset \bigcup_{n=0}^\f q\Lambda_n = q\Lambda.$$
Note that $\Lambda'$ is a spectrum of $\mu_{N,B}$.
Thus, $q\Lambda$ is a spectrum of $\mu_{N,B}$ if and only if $q\Lambda = \Lambda'$.
Therefore, we conclude that $q$ is a spectral eigenvalue of $(\mu_{N,B}, \Lambda)$ if and only if $q\Lambda = \Lambda'$.

We first show the sufficiency. Suppose that $(\ref{eq:eigenvalue-condition})$ holds. Then by (\ref{eq:Lambda-0}) and (\ref{eq:Lambda-0-prime}), we have $\Lambda_0' = q \Lambda$.
By recursion, we obtain $\Lambda'_n =  q \Lambda_n$ for all $n \in \N$.
It follows that $\Lambda' = q\Lambda$.
Thus, the number $q$ is a spectral eigenvalue of $(\mu_{N,B}, \Lambda)$.

Next, we prove the necessity.
Suppose that $q$ is a spectral eigenvalue of $(\mu_{N,B}, \Lambda)$.
Then we have $q\Lambda = \Lambda'$.
Take arbitrarily $$ x_0 \in K(N,L) \cap \frac{\Z}{qd}. $$
By (\ref{eq:Lambda-0}), we have $\Lambda_0 \subset (\Z/d)$.
By recursion, $\Lambda_n = N\Lambda_{n-1}+L \subset (\Z/d)$ for all $n \in \N$.
Thus, we obtain $\Lambda \subset (\Z/d)$.
By (\ref{eq:Lambda-0-prime}) we have $-qx_0 \in \Lambda_0' \subset \Lambda' = q \Lambda$.
It follows that $-x_0 \in \Lambda \subset (\Z/d)$, i.e., $x_0 \in (\Z/d)$.
So, \[ x_0 \in K(N,L) \cap \frac{\Z}{d}.\]
Since $x_0$ is arbitrary, we conclude that \[ \Big( K(N,L) \cap \frac{\Z}{qd} \Big) \subset \Big( K(N,L) \cap \frac{\Z}{d} \Big).  \]
The inverse inclusion is obvious.
Thus, we obtain (\ref{eq:eigenvalue-condition}) as desired.
\end{proof}

Recall the set $A_a(\delta)$ defined in (\ref{eq:A-a-delta}).
Next, we give a sufficient condition for the spectral pair $(\mu_{N,B},\Lambda)$ having infinitely many spectral eigenvalues.

\begin{proposition}\label{prop:large-eigenvalue}
  For $\dim_{\mathrm{H}} K(N,L) < \delta < 1$, if the set $A_N(\delta)$ is infinite, then there exists $p_0>0$, depending on $K(N,L)$ and $\delta$, such that every prime $p \in A_N(\delta) \cap [p_0,+\f)$ is a spectral eigenvalue of $(\mu_{N,B},\Lambda)$.
\end{proposition}
\begin{proof}
  Write $s=\dim_{\mathrm{H}} K(N,L)$ and let $c$ be the constant in Lemma \ref{lemma:period}.
  Choose a large enough integer $p_0 > N$ satisfying \[ p_0^{\delta-s} \ge c d^s.\]
  Take any prime $p \in  A_N(\delta) \cap [p_0,+\f)$. Then we have $p \ge p_0> N$ and
  \begin{equation}\label{eq:order-lower}
    \mathrm{Ord}_N(p) > p^\delta \ge p_0^{\delta-s} p^s \ge c d^s p^s.
  \end{equation}

  Suppose that $p$ is not a spectral eigenvalue of $(\mu_{N,B},\Lambda)$.
  Note that $\gcd(p,N) =1$. By Proposition \ref{prop:eigenvalue-equivalent}, we have
  \[ K(N, L) \cap \frac{\Z}{pd} \ne  K(N,L) \cap \frac{\Z}{d}. \]
  That is, \[ K(N, L) \cap \frac{\Z \setminus (p\Z) }{pd} \ne \emptyset. \]
  Take \[ x = \frac{v}{pu} \in K(N, L) \cap \frac{\Z \setminus (p\Z) }{pd}, \] where $v \in \Z$, $u \in \N$, $u \mid d$, and $\gcd(v,pu)=1$.
  Consider the self-similar set $K(N,L)$, and by Lemma \ref{lemma:period}, the number $x$ has an eventually periodic coding $\ell_{1} \ell_{2} \ldots \ell_{j}(\ell_{j+1}\ldots\ell_{j+m})^\f$ with each $\ell_k \in L$, and
  \begin{equation}\label{eq:order-upper}
    m \le c (p u)^s \le c d^s p^s.
  \end{equation}
  Note that
  \begin{align*}
    \frac{v}{pu} = x & = \sum_{k=1}^{j} \frac{\ell_k}{N^k} + \bigg( 1 + \frac{1}{N^m} + \frac{1}{N^{2m}} + \cdots \bigg) \sum_{k=1}^{m} \frac{\ell_{j+k}}{N^{j+k}} \\
    & = \sum_{k=1}^{j} \frac{\ell_k}{N^k} + \frac{N^m}{N^m - 1} \sum_{k=1}^{m} \frac{\ell_{j+k}}{N^{j+k}} \\
    & \in \frac{\Z}{N^j (N^m -1)}.
  \end{align*}
  Since $\gcd(p,v)=1$ and $\gcd(p,N)=1$, we have \[ p \mid (N^m-1),\;\text{i.e.},\; N^m \equiv 1 \pmod{p}. \]
  It follows that $\mathrm{Ord}_{N}(p) \mid m$.
  By (\ref{eq:order-upper}), we obtain \[ \mathrm{Ord}_{N}(p) \le m \le c d^s p^s, \]
  which contradicts (\ref{eq:order-lower}).
  Therefore, we conclude that $p$ is a spectral eigenvalue of $(\mu_{N,B},\Lambda)$.

  Since $p$ is taken arbitrarily, we complete the proof.
\end{proof}

Combined with results in number theory, Theorem \ref{thm-main-1} follows from Proposition \ref{prop:large-eigenvalue}.

\begin{proof}[Proof of Theorem \ref{thm-main-1}]
  (i) Since $\# B < N^{0.677}$, we have \[ s = \dim_{\mathrm{H}} K(N,L) \le \frac{\log \# L}{\log N} = \frac{\log \# B}{\log N} < 0.677. \]
  By Corollary \ref{cor:Baker-Harman}, the set $A_N(0.677)$ is infinite.
  By Proposition \ref{prop:large-eigenvalue}, every sufficiently large prime in $A_N(0.677)$ is a spectral eigenvalue of $(\mu_{N,B},\Lambda)$.
  The desired result follows directly from Corollary \ref{cor:Baker-Harman}.

  (ii) By using the same argument, we obtain the result from Proposition \ref{prop:large-eigenvalue} and Theorem \ref{thm:Erdos-1976}.
\end{proof}

Finally, we prove Theorem \ref{thm-main-2} by Theorem \ref{thm:finiteness} and Proposition \ref{prop:eigenvalue-equivalent}.

\begin{proof}[Proof of Theorem \ref{thm-main-2}]
  For $q \in \N_{\ge 2}$ with $\gcd(q,N)=1$, by Proposition \ref{prop:eigenvalue-equivalent}, $q$ is a spectral eigenvalue of $(\mu_{N,B}, \Lambda)$ if and only if \[ K(N, L) \cap \frac{\Z}{qd} = K(N,L) \cap \frac{\Z}{d}, \]
  which is equivalent to
  \begin{equation}\label{eq:eigenvalue-condition-2}
    K(N, dL) \cap \frac{\Z}{q} = K(N,dL) \cap \Z.
  \end{equation}

  Let $p=p_1 p_2 \ldots p_k$. Since $\#L = \# B < N$, we have $\dim_{\mathrm{H}} K(N,dL) < 1$. Note that $\gcd(p,N)=1$.
  By Theorem \ref{thm:finiteness}, we have \[ \#\big( D_p \cap K(N,dL) \big) < +\f. \]
  Note that $$D_p = \bigcup_{n=0}^\f \frac{\Z}{p^n},$$
  where the sequence $\{ \Z / p^n\}_{n=0}^\f$ is increasing with respect to the set inclusion.
  Thus, there exists $n_0 \in \N$ such that
  \begin{equation}\label{eq:n-0}
    K(N,dL) \cap D_p = K(N,dL) \cap \frac{\Z}{p^{n_0}}.
  \end{equation}

  Suppose that $(p_1 p_2 \ldots p_k)^{n_0} = p^{n_0}$ is a spectral eigenvalue of of $(\mu_{N,B}, \Lambda)$. Then by (\ref{eq:eigenvalue-condition-2}), we have
  \[ K(N, dL) \cap \frac{\Z}{p^{n_0}} = K(N,dL) \cap \Z. \]
  Together with (\ref{eq:n-0}), we conclude that
  \[ K(N,dL) \cap D_p = K(N,dL) \cap \Z. \]
  For any non-negative integers $n_1,n_2,\ldots,n_k$, observe that \[ \Z \subset \frac{\Z}{p_1^{n_1} p_2^{n_2} \cdots p_k^{n_k}} \subset D_p, \]
  and hence, \[ K(N,dL) \cap \frac{\Z}{p_1^{n_1} p_2^{n_2} \cdots p_k^{n_k}} = K(N,dL) \cap \Z.  \]
  By (\ref{eq:eigenvalue-condition-2}), we conclude that $p_1^{n_1} p_2^{n_2} \cdots p_k^{n_k}$ is a spectral eigenvalue of $(\mu_{N,B}, \Lambda)$.
\end{proof}

\section{Connections with number theory}\label{sec:connection}

\subsection{Artin's primitive root conjecture}

We first recall the content of Artin's primitive root conjecture and refer the reader to \cite{Moree-2012} for a survey.

\noindent
\textbf{Artin's primitive root conjecture}: \emph{For $a \in \N_{\ge 2}$ that is not a square number, there are infinitely many primes $p$ such that $\mathrm{Ord}_a(p) = p-1$.}

Suppose Artin's primitive root conjecture holds.
For any $N \in \N_{\ge 2}$, we can write $N = a^{2^k}$, where $a \in \N_{\ge 2}$ is not a square number and $k \in \N \cup \{0\}$.
Then there are infinitely many primes $p$ such that $\mathrm{Ord}_a(p) = p-1$. For such primes $p$, we have $\mathrm{Ord}_N(p) \ge (p-1)/2^k$.
Thus, for any $0<\delta <1$, the set $A_N(\delta)$ is infinite.
Applying Proposition \ref{prop:large-eigenvalue}, we obtain the following proposition.

\begin{proposition}\label{prop:Artin-conjecture}
Suppose Artin's primitive root conjecture holds.
Let $(\mu_{N,B}, \Lambda)$ be the canonical spectral pair generated by a Hadamard triple $(N,B,L)$ in $\R$ with $0\in B \cap L$ and $\# B < N$.
Then there are infinitely many prime spectral eigenvalues of $(\mu_{N,B},\Lambda)$.
\end{proposition}

It is worth mentioning that Heath-Brown \cite{Heath-Brown-1986} proved that there exists a constant $C>0$ such that \[ \# \big\{ a \le x: a \in \N_{\ge 2} \;\text{and Artin's primitive root conjecture fails for}\; a \big\} \le C (\log x)^2 . \]
This means that Artin's primitive root conjecture holds for almost all $a \in \N_{\ge 2}$.

\subsection{Elliott–Halberstam conjecture}

For $q \in \N_{\ge 2}$ and $a \in \Z$ with $\gcd(a,q)=1$, define
\[ \pi(x;q,a) = \# \big\{ p \le x: p \in \PP \;\;\text{and}\;\; p \equiv a \pmod{q} \big\}. \]
Dirichlet's theorem on primes in arithmetic progressions states that
\[ \pi(x;q,a) \sim \frac{\pi(x)}{\varphi(q)}\;\;\text{as}\; x \to +\f, \]
where $\varphi(n) = \# \big\{ 1\le k \le n: \gcd(k,n) =1 \big\}$ is Euler's totient function.
The Elliott–Halberstam conjecture is about the distribution of prime numbers in arithmetic progressions.

\noindent
\textbf{Elliott–Halberstam conjecture}: \emph{For $0 < \theta < 1$ and $A>0$, there exists a constant $C >0$ such that \[ \sum_{1 \le q \le x^\theta} \max_{y\le x}\max_{\gcd(a,q)=1} \bigg| \pi(y;q,a) - \frac{\pi(y)}{\varphi(q)} \bigg| \le \frac{C x}{(\log x)^A}\quad \forall x > 2. \]}

The Elliott–Halberstam conjecture was proved for $0<\theta < 1/2$, which is well-known as the Bombieri–Vinogradov theorem.
In \cite[Lemma 4.1]{WangZW-2018}, Wang showed that under the Elliott–Halberstam conjecture, for any $0<\delta< 1$, we have
\[ \lim_{x \to +\f} \frac{\#\big\{ p \le x: \; p \in \PP,\; P^+(p-1) > x^{\delta} \big\}}{\pi(x)} = 1- \rho\Big( \frac{1}{\delta} \Big),\]
where $\rho(u)$ is the Dickman function, defined as the unique continuous solution of the differential equation
\[
\begin{cases}
  \rho(u)=1, & 0\le u\le 1; \\
  u \rho'(u) + \rho(u-1) =0, & u \ge 1.
\end{cases}
\]
Note that $0< \rho(u) < 1$ for $u >1$.
Together with Lemma \ref{lemma:Goldfeld-1969}, we conclude that under the Elliott–Halberstam conjecture, for $a \in \N_{\ge 2}$ and $0<\delta< 1$, we have
\[ \liminf_{ x\to +\f} \frac{\#\big( A_a(\delta) \cap [0,x] \big)}{\pi(x)} >0. \]
Applying Proposition \ref{prop:large-eigenvalue}, we obtain the following proposition.

\begin{proposition}\label{prop:Elliott–Halberstam-conjecture}
Suppose the Elliott–Halberstam conjecture holds.
Let $(\mu_{N,B}, \Lambda)$ be the canonical spectral pair generated by a Hadamard triple $(N,B,L)$ in $\R$ with $0\in B \cap L$ and $\# B < N$.
Then there are infinitely many prime spectral eigenvalues of $(\mu_{N,B},\Lambda)$, and
$$\liminf_{x \to +\f} \frac{\# \big\{ p \le x: \; p \;\text{is a prime and a spectral eigenvalue of}\; (\mu_{N,B},\Lambda) \big\}}{\pi(x)} >0.$$
\end{proposition}

\begin{proof}[Proof of Theorem \ref{thm:conditional}]
  It follows from Propositions \ref{prop:Artin-conjecture} and \ref{prop:Elliott–Halberstam-conjecture}.
\end{proof}

\section*{Acknowledgements}
The author would like to thank Yuchen Ding for useful discussions on shifted primes with large prime factors, and thank Xing-Gang He for introducing Artin's primitive root conjecture.
The author was supported by the National Natural Science Foundation of China (No. 12501110, 12471085) and the China Postdoctoral Science Foundation (No. 2024M763857).

\bibliographystyle{abbrv}
\bibliography{SpectralMeasure,AnalysisNumberTheory}

@article {Erdos-1976,
    AUTHOR = {Erd\H{o}s, P.},
     TITLE = {Bemerkungen zu einer {A}ufgabe ({E}lem. {M}ath. {\bf 26}
              (1971), 43) by {G}. {J}aeschke},
   JOURNAL = {Arch. Math. (Basel)},
  FJOURNAL = {Archiv der Mathematik},
    VOLUME = {27},
      YEAR = {1976},
    NUMBER = {2},
     PAGES = {159--163},
      ISSN = {0003-889X,1420-8938},
   MRCLASS = {10H15},
  MRNUMBER = {404166},
MRREVIEWER = {S.\ Ikehara},
       DOI = {10.1007/BF01224655},
       URL = {https://doi.org/10.1007/BF01224655},
}

@article {Goldfeld-1969,
    AUTHOR = {Goldfeld, Morris},
     TITLE = {On the number of primes {$p$} for which {$p+a$} has a large
              prime factor},
   JOURNAL = {Mathematika},
  FJOURNAL = {Mathematika. A Journal of Pure and Applied Mathematics},
    VOLUME = {16},
      YEAR = {1969},
     PAGES = {23--27},
      ISSN = {0025-5793},
   MRCLASS = {10.42},
  MRNUMBER = {244176},
MRREVIEWER = {S.\ L.\ Segal},
       DOI = {10.1112/S0025579300004575},
       URL = {https://doi.org/10.1112/S0025579300004575},
}

@incollection {Baker-Harman-1995,
    AUTHOR = {Baker, R. C. and Harman, G.},
     TITLE = {The {B}run-{T}itchmarsh theorem on average},
 BOOKTITLE = {Analytic number theory, {V}ol. 1 ({A}llerton {P}ark, {IL},
              1995)},
    SERIES = {Progr. Math.},
    VOLUME = {138},
     PAGES = {39--103},
 PUBLISHER = {Birkh\"auser Boston, Boston, MA},
      YEAR = {1996},
      ISBN = {0-8176-3824-5},
   MRCLASS = {11N13 (11N36)},
  MRNUMBER = {1399332},
MRREVIEWER = {John\ B.\ Friedlander},
}

@article {Baker-Harman-1998,
    AUTHOR = {Baker, R. C. and Harman, G.},
     TITLE = {Shifted primes without large prime factors},
   JOURNAL = {Acta Arith.},
  FJOURNAL = {Acta Arithmetica},
    VOLUME = {83},
      YEAR = {1998},
    NUMBER = {4},
     PAGES = {331--361},
      ISSN = {0065-1036,1730-6264},
   MRCLASS = {11N25 (11N13 11N36)},
  MRNUMBER = {1610553},
MRREVIEWER = {John\ B.\ Friedlander},
       DOI = {10.4064/aa-83-4-331-361},
       URL = {https://doi.org/10.4064/aa-83-4-331-361},
}

@article {Moree-2012,
    AUTHOR = {Moree, Pieter},
     TITLE = {Artin's primitive root conjecture---a survey},
   JOURNAL = {Integers},
  FJOURNAL = {Integers},
    VOLUME = {12},
      YEAR = {2012},
    NUMBER = {6},
     PAGES = {1305--1416},
      ISSN = {1867-0652,1867-0660},
   MRCLASS = {11N37 (11A07 11B05)},
  MRNUMBER = {3011564},
       DOI = {10.1515/integers-2012-0043},
       URL = {https://doi.org/10.1515/integers-2012-0043},
}

@article {Heath-Brown-1986,
    AUTHOR = {Heath-Brown, D. R.},
     TITLE = {Artin's conjecture for primitive roots},
   JOURNAL = {Quart. J. Math. Oxford Ser. (2)},
  FJOURNAL = {The Quarterly Journal of Mathematics. Oxford. Second Series},
    VOLUME = {37},
      YEAR = {1986},
    NUMBER = {145},
     PAGES = {27--38},
      ISSN = {0033-5606,1464-3847},
   MRCLASS = {11A07 (11N13 11N35)},
  MRNUMBER = {830627},
MRREVIEWER = {D.\ J.\ Lewis},
       DOI = {10.1093/qmath/37.1.27},
       URL = {https://doi.org/10.1093/qmath/37.1.27},
}

@article {WangZW-2018,
    AUTHOR = {Wang, Zhiwei},
     TITLE = {Autour des plus grands facteurs premiers d'entiers
              cons\'ecutifs voisins d'un entier cribl\'e},
   JOURNAL = {Q. J. Math.},
  FJOURNAL = {The Quarterly Journal of Mathematics},
    VOLUME = {69},
      YEAR = {2018},
    NUMBER = {3},
     PAGES = {995--1013},
      ISSN = {0033-5606,1464-3847},
   MRCLASS = {11N36},
  MRNUMBER = {3859219},
MRREVIEWER = {O.\ Ramar\'e},
       DOI = {10.1093/qmath/hay010},
       URL = {https://doi.org/10.1093/qmath/hay010},
}

@article {An-Dong-He-2022,
    AUTHOR = {An, Li-Xiang and Dong, Xin-Han and He, Xing-Gang},
     TITLE = {On spectra and spectral eigenmatrix problems of the planar
              {S}ierpinski measures},
   JOURNAL = {Indiana Univ. Math. J.},
  FJOURNAL = {Indiana University Mathematics Journal},
    VOLUME = {71},
      YEAR = {2022},
    NUMBER = {2},
     PAGES = {913--952},
      ISSN = {0022-2518,1943-5258},
   MRCLASS = {42C05 (28A80)},
  MRNUMBER = {4420109},
MRREVIEWER = {Yan-Song\ Fu},
       DOI = {10.1512/iumj.2022.71.8873},
       URL = {https://doi.org/10.1512/iumj.2022.71.8873},
}

@article {An-Fu-Lai-2019,
    AUTHOR = {An, Li-Xiang and Fu, Xiao-Ye and Lai, Chun-Kit},
     TITLE = {On spectral {C}antor-{M}oran measures and a variant of
              {B}ourgain's sum of sine problem},
   JOURNAL = {Adv. Math.},
  FJOURNAL = {Advances in Mathematics},
    VOLUME = {349},
      YEAR = {2019},
     PAGES = {84--124},
      ISSN = {0001-8708,1090-2082},
   MRCLASS = {42C30 (26A30 28A80 42B10 42C05)},
  MRNUMBER = {3938848},
MRREVIEWER = {Ivan\ S.\ Yaroslavtsev},
       DOI = {10.1016/j.aim.2019.04.014},
       URL = {https://doi.org/10.1016/j.aim.2019.04.014},
}

@article {An-He-2014,
    AUTHOR = {An, Li-Xiang and He, Xing-Gang},
     TITLE = {A class of spectral {M}oran measures},
   JOURNAL = {J. Funct. Anal.},
  FJOURNAL = {Journal of Functional Analysis},
    VOLUME = {266},
      YEAR = {2014},
    NUMBER = {1},
     PAGES = {343--354},
      ISSN = {0022-1236,1096-0783},
   MRCLASS = {28A80 (42B05)},
  MRNUMBER = {3121733},
MRREVIEWER = {Manuel\ Mor\'{a}n},
       DOI = {10.1016/j.jfa.2013.08.031},
       URL = {https://doi.org/10.1016/j.jfa.2013.08.031},
}

@article {An-Lai-2023a,
    AUTHOR = {An, Li-Xiang and Lai, Chun-Kit},
     TITLE = {Product-form {H}adamard triples and its spectral self-similar
              measures},
   JOURNAL = {Adv. Math.},
  FJOURNAL = {Advances in Mathematics},
    VOLUME = {431},
      YEAR = {2023},
     PAGES = {Paper No. 109257, 41},
      ISSN = {0001-8708,1090-2082},
   MRCLASS = {42B10 (28A80 42C30)},
  MRNUMBER = {4630841},
       DOI = {10.1016/j.aim.2023.109257},
       URL = {https://doi.org/10.1016/j.aim.2023.109257},
}

@article {Chen-Cao-2025,
    AUTHOR = {Chen, Ming-Liang and Cao, Jian},
     TITLE = {Spectral structure and spectral eigenmatrix problem of planar
              self-similar measures with {$p$} digits},
   JOURNAL = {J. Geom. Anal.},
  FJOURNAL = {Journal of Geometric Analysis},
    VOLUME = {35},
      YEAR = {2025},
    NUMBER = {7},
     PAGES = {Paper No. 200, 34},
      ISSN = {1050-6926,1559-002X},
   MRCLASS = {28A80 (42C05)},
  MRNUMBER = {4910921},
       DOI = {10.1007/s12220-025-02037-w},
       URL = {https://doi.org/10.1007/s12220-025-02037-w},
}

@article {Chen-Liu-2023,
    AUTHOR = {Chen, Ming-Liang and Liu, Jing-Cheng},
     TITLE = {On spectra and spectral eigenmatrices of self-affine measures
              on {$\Bbb R^n$}},
   JOURNAL = {Bull. Malays. Math. Sci. Soc.},
  FJOURNAL = {Bulletin of the Malaysian Mathematical Sciences Society},
    VOLUME = {46},
      YEAR = {2023},
    NUMBER = {5},
     PAGES = {Paper No. 162, 18},
      ISSN = {0126-6705,2180-4206},
   MRCLASS = {28A25 (28A80 42C05)},
  MRNUMBER = {4616096},
       DOI = {10.1007/s40840-023-01559-2},
       URL = {https://doi.org/10.1007/s40840-023-01559-2},
}

@article {Chi-He-Wu-2025,
    AUTHOR = {Chi, Zi-Chao and He, Xing-Gang and Wu, Zhi-Yi},
     TITLE = {On a distinctive property of Fourier bases associated with $N$-Bernoulli Convolutions},
   JOURNAL = {arXiv:2506.00364},
  FJOURNAL = {arXiv:2506.00364},
      YEAR = {2025},
       URL = {https://arxiv.org/abs/2506.00364},
}

@article {Dai-2016,
    AUTHOR = {Dai, Xin-Rong},
     TITLE = {Spectra of {C}antor measures},
   JOURNAL = {Math. Ann.},
  FJOURNAL = {Mathematische Annalen},
    VOLUME = {366},
      YEAR = {2016},
    NUMBER = {3-4},
     PAGES = {1621--1647},
      ISSN = {0025-5831,1432-1807},
   MRCLASS = {42A65 (28A78 28A80 42B05)},
  MRNUMBER = {3563247},
MRREVIEWER = {Peter\ R.\ Massopust},
       DOI = {10.1007/s00208-016-1374-5},
       URL = {https://doi.org/10.1007/s00208-016-1374-5},
}

@ARTICLE{Dai-Fu-He-2026,
  author =       {Dai, Xin-Rong and Fu, Yan-Song and He, Xing-Gang},
  title =        {The theory of one-dimensional spectral measures (in Chinese)},
  journal =      {Sci. Sin. Math.},
  fjournal =     {SCIENTIA SINICA Mathematica},
  year =         {2026},
  volume =       {56},
  number =       {},
  pages =        {1--26},
  month =        {},
  note =         {},
  source =       {},
  doi =          {10.1360/SSM-2025-0083},
  url =          {https://doi.org/10.1360/SSM-2025-0083},
}

@article {Dutkay-Han-Sun-2014,
    AUTHOR = {Dutkay, Dorin Ervin and Han, Deguang and Sun, Qiyu},
     TITLE = {Divergence of the mock and scrambled {F}ourier series on
              fractal measures},
   JOURNAL = {Trans. Amer. Math. Soc.},
  FJOURNAL = {Transactions of the American Mathematical Society},
    VOLUME = {366},
      YEAR = {2014},
    NUMBER = {4},
     PAGES = {2191--2208},
      ISSN = {0002-9947,1088-6850},
   MRCLASS = {42B05 (28A80)},
  MRNUMBER = {3152727},
MRREVIEWER = {F.\ M\'{o}ricz},
       DOI = {10.1090/S0002-9947-2013-06021-7},
       URL = {https://doi.org/10.1090/S0002-9947-2013-06021-7},
}

@article {Dutkay-Haussermann-2016,
    AUTHOR = {Dutkay, Dorin Ervin and Haussermann, John},
     TITLE = {Number theory problems from the harmonic analysis of a
              fractal},
   JOURNAL = {J. Number Theory},
  FJOURNAL = {Journal of Number Theory},
    VOLUME = {159},
      YEAR = {2016},
     PAGES = {7--26},
      ISSN = {0022-314X,1096-1658},
   MRCLASS = {11A07 (11A51 42C30)},
  MRNUMBER = {3412709},
MRREVIEWER = {Anbhu\ Swaminathan},
       DOI = {10.1016/j.jnt.2015.07.009},
       URL = {https://doi.org/10.1016/j.jnt.2015.07.009},
}

@article {Dutkay-Haussermann-Lai-2019,
    AUTHOR = {Dutkay, Dorin Ervin and Haussermann, John and Lai, Chun-Kit},
     TITLE = {Hadamard triples generate self-affine spectral measures},
   JOURNAL = {Trans. Amer. Math. Soc.},
  FJOURNAL = {Transactions of the American Mathematical Society},
    VOLUME = {371},
      YEAR = {2019},
    NUMBER = {2},
     PAGES = {1439--1481},
      ISSN = {0002-9947,1088-6850},
   MRCLASS = {42B05 (28A25 42A85)},
  MRNUMBER = {3885185},
MRREVIEWER = {Xing-Gang\ He},
       DOI = {10.1090/tran/7325},
       URL = {https://doi.org/10.1090/tran/7325},
}

@incollection {Dutkay-Jorgensen-2008,
    AUTHOR = {Dutkay, Dorin Ervin and Jorgensen, Palle E. T.},
     TITLE = {Fourier series on fractals: a parallel with wavelet theory},
 BOOKTITLE = {Radon transforms, geometry, and wavelets},
    SERIES = {Contemp. Math.},
    VOLUME = {464},
     PAGES = {75--101},
 PUBLISHER = {Amer. Math. Soc., Providence, RI},
      YEAR = {2008},
      ISBN = {978-0-8218-4327-7},
   MRCLASS = {42C40 (28A80)},
  MRNUMBER = {2440130},
MRREVIEWER = {Joseph\ D.\ Lakey},
       DOI = {10.1090/conm/464/09077},
       URL = {https://doi.org/10.1090/conm/464/09077},
}

@article {Dutkay-Jorgensen-2012,
    AUTHOR = {Dutkay, Dorin Ervin and Jorgensen, Palle E. T.},
     TITLE = {Fourier duality for fractal measures with affine scales},
   JOURNAL = {Math. Comp.},
  FJOURNAL = {Mathematics of Computation},
    VOLUME = {81},
      YEAR = {2012},
    NUMBER = {280},
     PAGES = {2253--2273},
      ISSN = {0025-5718,1088-6842},
   MRCLASS = {47B32 (28A35 42B05)},
  MRNUMBER = {2945155},
MRREVIEWER = {Juan\ Matias\ Sepulcre},
       DOI = {10.1090/S0025-5718-2012-02580-4},
       URL = {https://doi.org/10.1090/S0025-5718-2012-02580-4},
}

@article {Dutkay-Kraus-2018,
    AUTHOR = {Dutkay, D. E. and Kraus, I.},
     TITLE = {Number theoretic considerations related to the scaling of
              spectra of {C}antor-type measures},
   JOURNAL = {Anal. Math.},
  FJOURNAL = {Analysis Mathematica},
    VOLUME = {44},
      YEAR = {2018},
    NUMBER = {3},
     PAGES = {335--367},
      ISSN = {0133-3852,1588-273X},
   MRCLASS = {11A07 (11A51 42C30)},
  MRNUMBER = {3854997},
MRREVIEWER = {Chun-Kit\ Lai},
       DOI = {10.1007/s10476-018-0505-5},
       URL = {https://doi.org/10.1007/s10476-018-0505-5},
}

@article {Dutkay-Lai-2017,
    AUTHOR = {Dutkay, Dorin Ervin and Lai, Chun-Kit},
     TITLE = {Spectral measures generated by arbitrary and random
              convolutions},
   JOURNAL = {J. Math. Pures Appl. (9)},
  FJOURNAL = {Journal de Math\'{e}matiques Pures et Appliqu\'{e}es.
              Neuvi\`eme S\'{e}rie},
    VOLUME = {107},
      YEAR = {2017},
    NUMBER = {2},
     PAGES = {183--204},
      ISSN = {0021-7824,1776-3371},
   MRCLASS = {42B10 (28A80 42C30)},
  MRNUMBER = {3597373},
MRREVIEWER = {Leandro\ Zuberman},
       DOI = {10.1016/j.matpur.2016.06.003},
       URL = {https://doi.org/10.1016/j.matpur.2016.06.003},
}

@article {Fuglede-1974,
    AUTHOR = {Fuglede, Bent},
     TITLE = {Commuting self-adjoint partial differential operators and a
              group theoretic problem},
   JOURNAL = {J. Functional Analysis},
  FJOURNAL = {Journal of Functional Analysis},
    VOLUME = {16},
      YEAR = {1974},
     PAGES = {101--121},
      ISSN = {0022-1236},
   MRCLASS = {47F05 (81.47)},
  MRNUMBER = {470754},
       DOI = {10.1016/0022-1236(74)90072-x},
       URL = {https://doi.org/10.1016/0022-1236(74)90072-x},
}

@article {Fu-He-2017,
    AUTHOR = {Fu, Yan-Song and He, Liu},
     TITLE = {Scaling of spectra of a class of random convolution on
              {$\Bbb{R}$}},
   JOURNAL = {J. Funct. Anal.},
  FJOURNAL = {Journal of Functional Analysis},
    VOLUME = {273},
      YEAR = {2017},
    NUMBER = {9},
     PAGES = {3002--3026},
      ISSN = {0022-1236,1096-0783},
   MRCLASS = {28A80 (42A16 42A85 42C05)},
  MRNUMBER = {3692329},
       DOI = {10.1016/j.jfa.2017.06.007},
       URL = {https://doi.org/10.1016/j.jfa.2017.06.007},
}

@article {Fu-He-Wen-2018,
    AUTHOR = {Fu, Yan-Song and He, Xing-Gang and Wen, Zhi-Xiong},
     TITLE = {Spectra of {B}ernoulli convolutions and random convolutions},
   JOURNAL = {J. Math. Pures Appl. (9)},
  FJOURNAL = {Journal de Math\'ematiques Pures et Appliqu\'ees. Neuvi\`eme
              S\'erie},
    VOLUME = {116},
      YEAR = {2018},
     PAGES = {105--131},
      ISSN = {0021-7824,1776-3371},
   MRCLASS = {42A85 (28A80 42B05 42C05)},
  MRNUMBER = {3826550},
MRREVIEWER = {Jan-Olav\ R\"onning},
       DOI = {10.1016/j.matpur.2018.06.002},
       URL = {https://doi.org/10.1016/j.matpur.2018.06.002},
}

@article {Fu-Tang-Wen-2022,
    AUTHOR = {Fu, Yan-Song and Tang, Min-Wei and Wen, Zhi-Ying},
     TITLE = {Convergence of mock {F}ourier series on generalized
              {B}ernoulli convolutions},
   JOURNAL = {Acta Appl. Math.},
  FJOURNAL = {Acta Applicandae Mathematicae},
    VOLUME = {179},
      YEAR = {2022},
     PAGES = {Paper No. 14, 23},
      ISSN = {0167-8019,1572-9036},
   MRCLASS = {42A20 (28A80 42A38)},
  MRNUMBER = {4431038},
MRREVIEWER = {Jing-Cheng\ Liu},
       DOI = {10.1007/s10440-022-00500-2},
       URL = {https://doi.org/10.1007/s10440-022-00500-2},
}

@article {He-Tang-Wu-2019,
    AUTHOR = {He, Xing-Gang and Tang, Min-Wei and Wu, Zhi-Yi},
     TITLE = {Spectral structure and spectral eigenvalue problems of a class
              of self-similar spectral measures},
   JOURNAL = {J. Funct. Anal.},
  FJOURNAL = {Journal of Functional Analysis},
    VOLUME = {277},
      YEAR = {2019},
    NUMBER = {10},
     PAGES = {3688--3722},
      ISSN = {0022-1236,1096-0783},
   MRCLASS = {42A65 (28A78 28A80 42B05)},
  MRNUMBER = {4001085},
MRREVIEWER = {Kazaros\ Kazarian},
       DOI = {10.1016/j.jfa.2019.05.019},
       URL = {https://doi.org/10.1016/j.jfa.2019.05.019},
}

@article {Hutchinson-1981,
    AUTHOR = {Hutchinson, John E.},
     TITLE = {Fractals and self-similarity},
   JOURNAL = {Indiana Univ. Math. J.},
  FJOURNAL = {Indiana University Mathematics Journal},
    VOLUME = {30},
      YEAR = {1981},
    NUMBER = {5},
     PAGES = {713--747},
      ISSN = {0022-2518,1943-5258},
   MRCLASS = {49F20 (00A69 28A12 58C27)},
  MRNUMBER = {625600},
MRREVIEWER = {F.\ J.\ Almgren, Jr.},
       DOI = {10.1512/iumj.1981.30.30055},
       URL = {https://doi.org/10.1512/iumj.1981.30.30055},
}

@article {Jiang-Lu-Wei-2024,
    AUTHOR = {Jiang, Mingxuan and Lu, Jian-Feng and Wei, Sai-Di},
     TITLE = {Spectral structure of a class of self-similar spectral
              measures with product form digit sets},
   JOURNAL = {Banach J. Math. Anal.},
  FJOURNAL = {Banach Journal of Mathematical Analysis},
    VOLUME = {18},
      YEAR = {2024},
    NUMBER = {4},
     PAGES = {Paper No. 66, 35},
      ISSN = {2662-2033,1735-8787},
   MRCLASS = {42C05 (28A80 42A65 42A85)},
  MRNUMBER = {4784890},
MRREVIEWER = {Zhi-Yi\ Wu},
       DOI = {10.1007/s43037-024-00368-4},
       URL = {https://doi.org/10.1007/s43037-024-00368-4},
}

@article {Jorgensen-Kornelson-Shuman-2011,
    AUTHOR = {Jorgensen, Palle E. T. and Kornelson, Keri A. and Shuman,
              Karen L.},
     TITLE = {Families of spectral sets for {B}ernoulli convolutions},
   JOURNAL = {J. Fourier Anal. Appl.},
  FJOURNAL = {The Journal of Fourier Analysis and Applications},
    VOLUME = {17},
      YEAR = {2011},
    NUMBER = {3},
     PAGES = {431--456},
      ISSN = {1069-5869,1531-5851},
   MRCLASS = {28A80 (28D05 42A16 42B10 42C25 46E30)},
  MRNUMBER = {2803943},
MRREVIEWER = {K.\ Gowri\ Navada},
       DOI = {10.1007/s00041-010-9158-x},
       URL = {https://doi.org/10.1007/s00041-010-9158-x},
}

@article {Jorgensen-Pedersen-1998,
    AUTHOR = {Jorgensen, Palle E. T. and Pedersen, Steen},
     TITLE = {Dense analytic subspaces in fractal {$L^2$}-spaces},
   JOURNAL = {J. Anal. Math.},
  FJOURNAL = {Journal d'Analyse Math\'{e}matique},
    VOLUME = {75},
      YEAR = {1998},
     PAGES = {185--228},
      ISSN = {0021-7670,1565-8538},
   MRCLASS = {46E30 (28A75 42C05 46L55 47B38)},
  MRNUMBER = {1655831},
MRREVIEWER = {Javier\ Soria},
       DOI = {10.1007/BF02788699},
       URL = {https://doi.org/10.1007/BF02788699},
}

@article {Kong-Li-Wang-2025,
    AUTHOR = {Kong, Derong and Li, Kun and Wang, Zhiqiang},
     TITLE = {Rational points in Cantor sets and spectral eigenvalue problem for self-similar spectral measures},
   JOURNAL = {arXiv:2503.22960},
  FJOURNAL = {arXiv:2503.22960},
      YEAR = {2025},
       URL = {https://arxiv.org/abs/2503.22960},
}

@article {Laba-Wang-2002,
    AUTHOR = {{\L}aba, Izabella and Wang, Yang},
     TITLE = {On spectral {C}antor measures},
   JOURNAL = {J. Funct. Anal.},
  FJOURNAL = {Journal of Functional Analysis},
    VOLUME = {193},
      YEAR = {2002},
    NUMBER = {2},
     PAGES = {409--420},
      ISSN = {0022-1236,1096-0783},
   MRCLASS = {28A80},
  MRNUMBER = {1929508},
MRREVIEWER = {Henning\ Fernau},
       DOI = {10.1006/jfan.2001.3941},
       URL = {https://doi.org/10.1006/jfan.2001.3941},
}

@article {Li-2011,
    AUTHOR = {Li, Jian-Lin},
     TITLE = {Spectra of a class of self-affine measures},
   JOURNAL = {J. Funct. Anal.},
  FJOURNAL = {Journal of Functional Analysis},
    VOLUME = {260},
      YEAR = {2011},
    NUMBER = {4},
     PAGES = {1086--1095},
      ISSN = {0022-1236,1096-0783},
   MRCLASS = {42B10 (28A80)},
  MRNUMBER = {2747015},
MRREVIEWER = {Andrei\ K.\ Lerner},
       DOI = {10.1016/j.jfa.2010.12.001},
       URL = {https://doi.org/10.1016/j.jfa.2010.12.001},
}

@article {Li-Xing-2017,
    AUTHOR = {Li, Jian-Lin and Xing, Dan},
     TITLE = {Multiple spectra of {B}ernoulli convolutions},
   JOURNAL = {Proc. Edinb. Math. Soc. (2)},
  FJOURNAL = {Proceedings of the Edinburgh Mathematical Society. Series II},
    VOLUME = {60},
      YEAR = {2017},
    NUMBER = {1},
     PAGES = {187--202},
      ISSN = {0013-0915,1464-3839},
   MRCLASS = {28A80 (42A65 42C05)},
  MRNUMBER = {3589848},
MRREVIEWER = {Qi-Rong\ Deng},
       DOI = {10.1017/S0013091515000565},
       URL = {https://doi.org/10.1017/S0013091515000565},
}

@article {Li-Miao-Wang-2022,
    AUTHOR = {Li, Wenxia and Miao, Jun Jie and Wang, Zhiqiang},
     TITLE = {Weak convergence and spectrality of infinite convolutions},
   JOURNAL = {Adv. Math.},
  FJOURNAL = {Advances in Mathematics},
    VOLUME = {404},
      YEAR = {2022},
     PAGES = {Paper No. 108425, 26},
      ISSN = {0001-8708,1090-2082},
   MRCLASS = {28A80 (42C30 60B10)},
  MRNUMBER = {4416138},
MRREVIEWER = {Yan-Song\ Fu},
       DOI = {10.1016/j.aim.2022.108425},
       URL = {https://doi.org/10.1016/j.aim.2022.108425},
}

@article {Li-Wu-2022,
    AUTHOR = {Li, Jin-Jun and Wu, Zhi-Yi},
     TITLE = {On spectral structure and spectral eigenvalue problems for a
              class of self similar spectral measure with product form},
   JOURNAL = {Nonlinearity},
  FJOURNAL = {Nonlinearity},
    VOLUME = {35},
      YEAR = {2022},
    NUMBER = {6},
     PAGES = {3095--3117},
      ISSN = {0951-7715,1361-6544},
   MRCLASS = {28A80 (42C05)},
  MRNUMBER = {4443929},
MRREVIEWER = {Hai-Xiong\ Li},
       DOI = {10.1088/1361-6544/ac6b0c},
       URL = {https://doi.org/10.1088/1361-6544/ac6b0c},
}

@article {Liu-Liu-Tang-Wu-2024,
    AUTHOR = {Liu, Jing-Cheng and Liu, Ming and Tang, Min-Wei and Wu, Sha},
     TITLE = {On spectral eigenmatrix problem for the planar self-affine
              measures with three digits},
   JOURNAL = {Ann. Funct. Anal.},
  FJOURNAL = {Annals of Functional Analysis},
    VOLUME = {15},
      YEAR = {2024},
    NUMBER = {4},
     PAGES = {Paper No. 83, 22},
      ISSN = {2639-7390,2008-8752},
   MRCLASS = {28A25 (28A80 42C05 46C05)},
  MRNUMBER = {4791481},
       DOI = {10.1007/s43034-024-00386-1},
       URL = {https://doi.org/10.1007/s43034-024-00386-1},
}

@article {Liu-Tang-Wu-2023,
    AUTHOR = {Liu, Jing-Cheng and Tang, Min-Wei and Wu, Sha},
     TITLE = {The spectral eigenmatrix problems of planar self-affine
              measures with four digits},
   JOURNAL = {Proc. Edinb. Math. Soc. (2)},
  FJOURNAL = {Proceedings of the Edinburgh Mathematical Society. Series II},
    VOLUME = {66},
      YEAR = {2023},
    NUMBER = {3},
     PAGES = {897--918},
      ISSN = {0013-0915,1464-3839},
   MRCLASS = {28A25 (28A80 42C05 46C05)},
  MRNUMBER = {4637402},
       DOI = {10.1017/s0013091523000469},
       URL = {https://doi.org/10.1017/s0013091523000469},
}

@article {Lu-Dong-Zhang-2022,
    AUTHOR = {Lu, Zheng-Yi and Dong, Xin-Han and Zhang, Peng-Fei},
     TITLE = {Spectrality of some one-dimensional {M}oran measures},
   JOURNAL = {J. Fourier Anal. Appl.},
  FJOURNAL = {The Journal of Fourier Analysis and Applications},
    VOLUME = {28},
      YEAR = {2022},
    NUMBER = {4},
     PAGES = {Paper No. 63, 22},
      ISSN = {1069-5869,1531-5851},
   MRCLASS = {42C05 (28A80)},
  MRNUMBER = {4450149},
MRREVIEWER = {Zhi-Yi\ Wu},
       DOI = {10.1007/s00041-022-09954-2},
       URL = {https://doi.org/10.1007/s00041-022-09954-2},
}

@article {Pan-Ai-2023,
    AUTHOR = {Pan, Wu-Yi and Ai, Wen-Hui},
     TITLE = {Divergence of mock {F}ourier series for spectral measures},
   JOURNAL = {Proc. Roy. Soc. Edinburgh Sect. A},
  FJOURNAL = {Proceedings of the Royal Society of Edinburgh. Section A.
              Mathematics},
    VOLUME = {153},
      YEAR = {2023},
    NUMBER = {6},
     PAGES = {1818--1832},
      ISSN = {0308-2105,1473-7124},
   MRCLASS = {28A80 (42A20 42B05)},
  MRNUMBER = {4666060},
       DOI = {10.1017/prm.2022.68},
       URL = {https://doi.org/10.1017/prm.2022.68},
}

@article {Strichartz-2000,
    AUTHOR = {Strichartz, Robert S.},
     TITLE = {Mock {F}ourier series and transforms associated with certain
              {C}antor measures},
   JOURNAL = {J. Anal. Math.},
  FJOURNAL = {Journal d'Analyse Math\'{e}matique},
    VOLUME = {81},
      YEAR = {2000},
     PAGES = {209--238},
      ISSN = {0021-7670,1565-8538},
   MRCLASS = {42A38 (28A75 42C05)},
  MRNUMBER = {1785282},
MRREVIEWER = {Steen\ Pedersen},
       DOI = {10.1007/BF02788990},
       URL = {https://doi.org/10.1007/BF02788990},
}

@article {Strichartz-2006,
    AUTHOR = {Strichartz, Robert S.},
     TITLE = {Convergence of mock {F}ourier series},
   JOURNAL = {J. Anal. Math.},
  FJOURNAL = {Journal d'Analyse Math\'{e}matique},
    VOLUME = {99},
      YEAR = {2006},
     PAGES = {333--353},
      ISSN = {0021-7670,1565-8538},
   MRCLASS = {42B05 (42A20)},
  MRNUMBER = {2279556},
MRREVIEWER = {Gy\"{o}rgy\ G\'{a}t},
       DOI = {10.1007/BF02789451},
       URL = {https://doi.org/10.1007/BF02789451},
}

@article {WuH-2024,
    AUTHOR = {Wu, Hai-Hua},
     TITLE = {Spectral self-similar measures with alternate contraction
              ratios and consecutive digits},
   JOURNAL = {Adv. Math.},
  FJOURNAL = {Advances in Mathematics},
    VOLUME = {443},
      YEAR = {2024},
     PAGES = {Paper No. 109585, 33},
      ISSN = {0001-8708,1090-2082},
   MRCLASS = {28A80 (42B10 42C05)},
  MRNUMBER = {4712262},
MRREVIEWER = {Ming-Liang\ Chen},
       DOI = {10.1016/j.aim.2024.109585},
       URL = {https://doi.org/10.1016/j.aim.2024.109585},
}

@article {Yi-Zhang-2025a,
    AUTHOR = {Yi, Shan-Feng and Zhang, Min-Min},
     TITLE = {Scaling spectrum of a class of self-similar measures with
              product form on {$\Bbb R$}},
   JOURNAL = {Forum Math.},
  FJOURNAL = {Forum Mathematicum},
    VOLUME = {37},
      YEAR = {2025},
    NUMBER = {3},
     PAGES = {765--776},
      ISSN = {0933-7741,1435-5337},
   MRCLASS = {42C05 (28A80 42A65)},
  MRNUMBER = {4857542},
       DOI = {10.1515/forum-2023-0466},
       URL = {https://doi.org/10.1515/forum-2023-0466},
}

\end{document}